\documentclass[a4paper,12pt]{article}

\usepackage{amscd}
\usepackage{amsfonts}
\usepackage{amsmath}
\usepackage{amssymb}
\usepackage{amsthm}
\usepackage[T1]{fontenc} 
\usepackage{here} 
\usepackage{mathrsfs} 
\usepackage{txfonts} 
\usepackage[all]{xy} 
\usepackage{algorithm} 
\usepackage{algpseudocode} 
\usepackage{diagbox} 

\allowdisplaybreaks

\theoremstyle{plain}
\newtheorem{thm}{Theorem}[section]
\newtheorem{lmm}[thm]{Lemma}
\newtheorem{prp}[thm]{Proposition}
\newtheorem{crl}[thm]{Corollary}

\theoremstyle{definition}
\newtheorem{dfn}[thm]{Definition}

\newtheorem{exm}[thm]{Example}

\newcommand{\vs}[1][0.2]{\vspace{#1in}\noindent\ignorespaces}
\newcommand{\ba}{\begin{array*}}
\newcommand{\ea}{\end{array*}}
\newcommand{\be}{\begin{eqnarray*}}
\newcommand{\ee}{\end{eqnarray*}}
\newcommand{\bi}{\begin{itemize}}
\newcommand{\ei}{\end{itemize}}
\newcommand{\bb}{\vs\begin{itembox}}
\newcommand{\eb}{\end{itembox}}
\newcommand{\bc}{\begin{center}}
\newcommand{\ec}{\end{center}}
\newcommand{\bs}{\vs\begin{screen}}
\newcommand{\es}{\end{screen}}

\def\ens#1{{\mathchoice{\left\{ #1 \right\}}{\{ #1 \}}{\{ #1 \}}{\{ #1 \}}}}
\def\set#1#2{{\mathchoice{\left\{ #1 \ \middle| \ #2 \right\}}{\{ #1 \mid #2 \}}{\{ #1 \mid #2 \}}{\{ #1 \mid #2 \}}}}
\def\r#1{\text{\rm #1}}

\def\Bigv#1{\left| #1 \right|}
\def\v#1{{\mathchoice{\Bigv{#1}}{| #1 |}{| #1 |}{| #1 |}}}

\algnewcommand\algorithmicbreak{{\bf break}}
\algnewcommand\Break{\algorithmicbreak{}}
\algnewcommand\algorithmiccontinue{{\bf continue}}
\algnewcommand\Continue{\algorithmiccontinue{}}

\newcommand{\bN}{\mathbb{N}}

\newcommand{\bP}{\mathbb{P}}
\newcommand{\bQ}{\mathbb{Q}}
\newcommand{\bR}{\mathbb{R}}

\newcommand{\bZ}{\mathbb{Z}}

\newcommand{\cA}{\mathscr{A}}

\newcommand{\N}{\bN}
\newcommand{\Q}{\bQ}
\newcommand{\R}{\bR}
\newcommand{\Z}{\bZ}

\newcommand{\Fp}{\mathbb{F}_p}

\newcommand{\DF}{\r{DF}}
\newcommand{\REDC}{\r{REDC}}

\title{Transcendence of Montgomery Reduction Factor in the Ring of Integers Modulo Infinitely Large Primes}
\author{Tomoki Mihara}
\date{}

\begin{document}

\maketitle
\begin{abstract}
We prove the transcendence over $\mathbb{Q}$ of the image of Montgomery reduction factor $R'$ in the ring $\mathscr{A}$ of integers modulo infinitely large primes. Here, for an odd prime number $p$, $R'$ is defined as the modular inverse of a power $R$ of $2$ modulo $p$ satisfying $2^{-k} R \leq p < R$ for a fixed constant $k \in \mathbb{N}_{> 0}$ typically given as the standard bit size $32$ of an integer type, and is the element of $\mathbb{Z}/p \mathbb{Z}$ representing Montgomery reduction regarded as a $\mathbb{Z}/p \mathbb{Z}$-linear homomorphism.
\end{abstract}

\tableofcontents

\section{Introduction}
\label{Introduction}

Let $(N,R) \in \N_{> 0}^2$ with $N < R$ and $\gcd(N,R) = 1$. Montgomery reduction $\REDC$ (cf.\ \cite{Mon85}) is an algorithm which rapidly computes the multiplication by the modular inverse of $R$ modulo $N$. The significance of $\REDC$ is its explicit implementation using division by $R$ rather than $N$. Let $(N',R') \in \N_{< R} \times \N_{< N}$ with $RR' - NN' = 1$. Such a pair is found by Euclidean algorithm with division by $R$. Here is a pseudocode of $\REDC$ for $(N,R)$, which we denote by $\REDC_{N,R}$:

\begin{figure}[H]
\begin{algorithm}[H]
\caption{Montgomery reduction for $(N,R)$:}
\begin{algorithmic}[1]
\Function {REDC}{$i$}
	\State $m \gets (i \bmod R)N' \bmod R$
	\State $i' \gets (i + mN)/R$
	\If {$i' < N$}
		\State \Return $i$
	\Else
		\State \Return $i - N$
	\EndIf
\EndFunction
\end{algorithmic}
\end{algorithm}
\end{figure}

If $N$ is an odd number and $R$ is chosen as a power of $2$, then the division by $R$ is inexpensive to proceed in computers: the quotient is computed by the right bit shift operation, and the residue is computed by the bitwise AND operation with $R - 1$. Since naive computation of modular arithmetic modulo $N$ requires much machine operations, Montgomery reduction performs better in algorithms requiring extensive modular arithmetic. Although there are other algorithms for modular arithmetic, we focus only on Montgomery reduction in this study.

\vs
We denote by $\bP$ the set of prime numbers. The aim of this paper is to study the algebraic complexity of the Montgomery reduction factor $R'$ for the case $N$ is an odd prime number and $R$ is a power of $2$ satisfying $2^{-k} R \leq N < R$ for a fixed constant $k \in \N_{> 0}$ typically given as the standard bit size $32$ of an integer type, from a new point of view using the ring
\be
\cA \coloneqq \prod_{p \in \bP} \Fp \bigg/ \bigoplus_{p \in \bP} \Fp
\ee
of integers modulo infinitely large primes. In number theory, $\cA$ plays a role of the base ring in the study of finite multiple zeta values, which are finite analogues of multiple zeta values introduced by D.\ Zagier. Starting from \cite{KZ}, there have been various studies on specific elements of $\cA$. See \cite{Mih26-1} \S 1--2 for algebraic properties of $\cA$. We give a new transcendence result on $R'$:

\begin{thm}
Let $R' \colon \bP \to \N$. Suppose that there exists a $k \in \N_{> 0}$ such that for all but finitely many $p \in \bP$, $R'(p)$ is a modular inverse of some power $R$ of $r$ modulo $p$ with $2^{-k}R \leq p < R$, then the image of $R'$ in $\cA$ is transcendental over $\Q$.
\end{thm}

We will show a slightly stronger result in Theorem \ref{Montgomery}. As a consequence, we obtain that the map $\bP \to \N$ assigning to each $p \in \bP$ the bit flip of $p$ defines an element of $\cA$ transcendental over $\Q$ (cf.\ Corollary \ref{digit flip} and the paragraph after the proof).

\vs
We explain motivation and preceding studies. It is well-known (cf.\ \cite{Rob63} (v)) that the image of Fibonacci sequence $(F_n)_{n=1}^{\infty}$ in $\cA$ coincides with that of the Dirichlet character associated to Legendre symbol $(\frac{\cdot}{5})$, and hence is a zero of $x^2 - 1 \in \Q[x]$ in $\cA$. J.\ Rosen generalised this fact in \cite{Ros18} Theorem 1.4 by characterising the subset of $\cA$ consisting of the images of $\Q$-linear recurrent sequences in \cite{Ros18} Theorem 1.1. As a result, the image in $\cA$ of a $\Q$-linear recurrent sequence can never be transcendental over $\Q$.

\vs
As an extension of \cite{Ros18} Theorem 1.4, we proved in \cite{Mih26-2} Theorem 2.6 that the image in $\cA$ of the composite of a $\Q$-linear recurrent sequence and a polynomial function over $\N$ can never be transcendental over $\Q$. On the other hand, we proved in \cite{Mih26-2} Theorem 3.10 that a non-trivial double exponential function defines an element of $\cA$ transcendental over $\Q$ under a number theoretic conjecture. This implies that the image in $\cA$ of the composite of two $\N$-linear recurrent sequences can be transcendental over $\Q$ under the same conjecture.

\vs
Through the observation, we regarded transcendence in $\cA$ over $\Q$ as a new algebraic measure of complexity of an arithmetic algorithm independent of the implementation. For example, if we consider an algorithm to compute a square root of a fixed constant $n \in \Z$ modulo $p \in \bP$ when $n$ is a quadratic residue modulo $p$ and to return a constant $c$ otherwise, then the algorithm defines $f \in \cA$ satisfying the algebraic relation $(f^2 - n)(f - c) = 0$ over $\Q$. In particular, Tonelli--Shanks algorithm is used to compute a square root of $n$ modulo $p \in \bP$ as the product of the main term $n^{\frac{k_p + 1}{2}}$ modulo $p$ and a correction term given as a power of a quadratic non-residue modulo $p$, where $k_p$ denotes the greatest odd divisor of $p - 1$.

\vs
Although the return values of Tonelli--Shanks algorithm which returns a constant when $n$ is a quadratic non-residue modulo $p$ can never define an element of $\cA$ transcendental over $\Q$, we proved in \cite{Mih26-3} Theorem 4.2 and \cite{Mih26-3} Corollary 4.6 that the two factors, i.e.\ the main term and the correction term, always define elements of $\cA$ transcendental over $\Q$. This means that the correction term can not be ignored for infinitely many $p \in \bP$, and explains the difference of the main term and an actual square root modulo $p$.

\vs
As the example of Fibonacci sequence shows, the property that the return values of (or more generally, intermediate terms appearing in) an arithmetic algorithm can never define an element of $\cA$ transcendental over $\Q$ indicates the possibility to reduce computational complexity when we consider modular arithmetic modulo $p \in \bP$. Therefore, we observe when an arithmetic algorithm returns a value which defines an element of $\cA$ transcendental over $\Q$.

\vs
We briefly explain contents of this paper. In \S \ref{Convention}, we introduce convention in this paper. In \S \ref{Sparse Rounding}, we give a criterion of transcendence in $\cA$ over $\Q$ which we use in this paper. In \S \ref{Montgomery Reduction}, we show Theorem \ref{Montgomery} mentioned above.

\section{Convention}
\label{Convention}

We denote by $\N$ the set of non-negative integers, and by $\bP$ the set of prime numbers. For $p \in \bP$, we denote by $\Fp$ the finite field $\Z/p \Z$. We set
\be
\cA \coloneqq \prod_{p \in \bP} \Fp \bigg/ \bigoplus_{p \in \bP} \Fp.
\ee
For a set $I$, we denote by $\# I$ its cardinality. For sets $X$ and $Y$, we denote by $X^Y$ the set of maps $Y \to X$. When we handle a sequence $s$ indexed by a set $I$, we frequently use the map notation $s(i)$ instead of the subscript notation $s_i$ to point the entry at $i \in I$, in order to avoid massive use of subscripts. For a set $X$, $x \in X$, and a binary relation $R$ on $X$, we set $X_{R x} \coloneqq \set{x' \in X}{x' R x}$. For example, every $d \in \N$ is identical to $\N_{< d} = \set{i \in \N}{i < d}$, and hence for a set $X$, $X^d$ formally means $X^{\N_{< d}}$, which is naturally identified with the set of $d$-tuples in $X$. For $(s,t) \in \R^X \times \R^Y$ with $X,Y \subset \R$, we define
\be
s \in o(t) & \stackrel{\r{def}}{\Leftrightarrow} & \forall c \in \R_{> 0}[\exists x_0 \in X \cap Y[\forall x \in (X \cap Y)_{\geq x_0}[\v{s(x)} \leq ct(x)]]].
\ee
For $a \in \N^{\bP}$, we call
\be
(a(p) + p \Z)_{p \in \bP} + \bigoplus_{p \in \bP} \Fp \in \cA
\ee
{\it the image of $a$ in $\cA$}. We say that $a \in \cA$ is {\it naively transcendental} if there exists no $f \in \Q[x] \setminus \ens{0}$ such that $f(a) = 0$. Here, we formally define $0^0 \coloneqq 1$.

\section{Sparse Rounding}
\label{Sparse Rounding}

Let $S$ be an infinite subset of $\N$. We denote by $e_S \colon \N \to S$ the enumeration function, i.e.\ the function characterised by the recursive relation
\be
e_S(i) \coloneqq \min (S \setminus e_S(\N_{< i})).
\ee
We formally set $e_S(i) \coloneqq e_S(0)$ for $i \in \Z_{< 0}$. We study a rounding of prime numbers into $S$.

\begin{dfn}
Let $k \in \N$. We call $f \in \N^{\bP}$ an {\it $(S,k)$-rounding map} if for any $p \in \bP$, there exists some $i \in \N$ such that $f(p) = e_S(i)$ and $e_S(i-k) \leq p \leq e_S(i+k)$, and denote by $C_{S,k} \subset \N^{\bP}$ the subset of $(S,k)$-rounding maps.
\end{dfn}

As a variant of \cite{MS26} Lemma 2.5 (cf.\ proof of \cite{LW25} Theorem 1), we give a criterion of the naive transcendence of an $(S,k)$-rounding in terms of the growth rate of $e_S$.

\begin{prp}
\label{sparse}
Let $k \in \N$. If $x \log e_S(x) \log e_S(x-k) \in o(e_S(x-k))$, then the image of any $f \in C_{S,k}$ in $\cA$ is naively transcendental.
\end{prp}

\begin{proof}
Let $F \in \Q[x] \setminus \ens{0}$. We show that the image of $F(f)$ in $\cA$ is not $0$. Multiplying $F$ by the least common multiple of denominators of coefficients of $F$, we may assume $F \in \Z[x] \setminus \ens{0}$. By $F \neq 0$ and $F(x) \in o(x^{\deg(F)+1})$, there exists some $q_0 \in \Z$ such that for any $q \in \Z_{\geq q_0}$,
\be
0 < \v{F(q)} < q^{\deg(F) + 1}
\ee
holds. We denote by $i_0$ the least $i \in \N$ such that $q_0 \leq e_S(i)$ and $1 < e_S(i-k)$. Let $i \in \N_{\geq i_0}$. Set
\be
P_i & \coloneqq & \bigcup_{i'=i_0}^{i} \set{p \in \bP}{F(e_s(i')) \equiv 0 \pmod{p}} \\
P'_i & \coloneqq & \bP \cap [e_S(i_0+k),e_S(i-k)].
\ee
For any $i' \in \N \cap [i_0,i]$, we have $F(e_S(i')) \in \Z \setminus \ens{0}$ by the definition of $q_0$ and $i_0$, and hence
\be
& & \# \set{p \in \bP}{F(e_s(i')) \equiv 0 \pmod{p}} \leq \log_2 F(e_S(i')) \leq \log_2 F(e_S(i)) \\
& < & \log_2 e_S(i)^{\deg(F) + 1} = (\deg(F) + 1) \log_2 e_S(i).
\ee
This implies
\be
\# P_i \leq (i-i_0+1) \times (\deg(F) + 1) \log_2 e_S(i) \leq (\deg(F) + 1)i \log_2 e_S(i).
\ee
For any $p \in P'_i$, we have $f(p) = e_S(i')$ for some $i' \in \N \cap [i_0,i]$ by $f \in C_{S,k}$, and hence $F(f(p)) \equiv 0 \pmod{p}$ is equivalent to $p \in P_i$. We have
\be
& & \# \bP_{\leq e_S(i-k)} \leq \# P'_i + e_S(i_0+k) \leq (\#(P'_I \setminus P_i) + \# P_i) + e_S(i_0+k) \\
& \leq & \#(P'_I \setminus P_i) + (\deg(F) + 1)i \log_2 e_S(i) + e_S(i_0+k),
\ee
and hence
\be
& & \#(P'_i \setminus P_i) \\
& \geq & \# \bP_{\leq e_S(i-k)} - ((\deg(F) + 1)i \log_2 e_S(i) + e_S(i_0+k)) \\
& \geq & \frac{e_S(i-k)}{\log e_S(i-k)} \left( \frac{\# \bP_{\leq e_S(i-k)} \log e_S(i-k)}{e_S(i-k)} - \frac{(\deg(F) + 1)i (\log_2 e_S(i) + e_S(i_0+k)) \log e_S(i-k)}{e_S(i-k)} \right) \\
& \stackrel{i \to \infty}{\longrightarrow} & \infty
\ee
by prime number theorem and $x \log e_S(x) \log e_S(x-k) \in o(e_S(x-k))$. This implies $F(f(p)) \not\equiv 0 \pmod{p}$ for infinitely many $p \in \bP$, i.e.\ the image of $F(f)$ in $\cA$ is not $0$.
\end{proof}

\begin{crl}
\label{floor}
Suppose $x \log e_S(x) \log e_S(x-1) \in o(e_S(x-1))$.
\bi
\item[(i)] If $S_{\leq 2} \neq \emptyset$ (resp.\ $S_{< 2} \neq \emptyset$), then the image of $(\max S_{\leq p})_{p \in P} \in \N^{\bP}$ (resp.\ $(\max S_{< p})_{p \in P} \in \N^{\bP}$) in $\cA$ is naively transcendental.
\item[(ii)] The image of $(\min S_{\geq p})_{p \in \bP} \in \N^{\bP}$ (resp.\ $(\min S_{> p})_{p \in \bP} \in \N^{\bP}$) in $\cA$ is naively transcendental.
\item[(iii)] Let $f \colon \bP \to \N$ be a map assigning to $p \in \bP$ one of the nearest elements of $S$. Then the image of $f$ in $\cA$ is naively transcendental.
\ei
\end{crl}

We note that the condition in (i) is just formally given and is not so meaningful, because the value at $p = 2$ does not affect the image in $\cA$.

\begin{proof}
The assertions immediately follows from Proposition \ref{sparse} applied to $k = 1$.
\end{proof}

We give an example of an application of Corollary \ref{floor}, which is essentially the same as \cite{MS26} Example 2.6 given as an example of an application of \cite{MS26} Lemma 2.5:

\begin{exm}
Let $e \in \N_{> 1}$. The image of $(\lfloor \sqrt[e]{p} \rfloor^e)_{p \in \bP} \in \N^{\bP}$ in $\cA$ is naively transcendental by Corollary \ref{floor} (i) applied to the case where $S$ is the set of perfect $e$-th powers, and hence so is that of $(\lfloor \sqrt[e]{p} \rfloor)_{p \in \bP} \in \N^{\bP}$. Similarly, the image of $(\lceil \sqrt[e]{p} \rceil^e)_{p \in \bP} \in \N^{\bP}$ in $\cA$ is naively transcendental by Corollary \ref{floor} (ii), and hence so is that of $(\lceil \sqrt[e]{p} \rceil)_{p \in \bP} \in \N^{\bP}$.
\end{exm}

We give another example, which is not directly derived from \cite{MS26} Lemma 2.5:

\begin{exm}
For $p \in \bP$, we denote by $s_p \in \N$ the least perfect number greater than $p$ (resp.\ the greatest perfect power less than $p$). Then the image of $(s_p)_{p \in \bP} \in \N^{\bP}$ in $\cA$ is naively transcendental. Indeed, let $S \subset \N$ denote the subset of perfect powers. We have
\be
\# S_{\leq x} & \leq & 1 + \sum_{e=2}^{\infty} \# \set{i \in \N_{\leq x}}{\sqrt[e]{i} \in \N_{> 1}} \\
& = & 1 + \sum_{k=1}^{\infty} k \# \set{e \in \N_{> 1}}{\# \set{i \in \N_{\leq x}}{\sqrt[e]{i} \in \N_{> 1}} = k} \\
& = & 1 + \sum_{k=2}^{\infty} (k - 1) \# \set{e \in \N_{> 1}}{\# \set{i \in \N_{\leq x}}{\sqrt[e]{i} \in \N} = k} \\
& = & 1 + \sum_{k=2}^{\infty} \# \set{e \in \N_{> 1}}{\# \set{i \in \N_{\leq x}}{\sqrt[e]{i} \in \N} \geq k} \\
& = & 1 + \sum_{k=2}^{\infty} \# \set{e \in \N_{> 1}}{\sqrt[e]{x} \geq k} \\
& = & 1 + \sum_{k=2}^{\infty} \# \set{e \in \N_{> 1}}{x \geq k^e} \\
& = & 1 + \sum_{k=2}^{\infty} \# \set{e \in \N_{> 1}}{\log_k x \geq e} \\
& \leq & 1 + \sum_{k=2}^{\lfloor \sqrt{x} \rfloor} (\log_k x) - 1 \\
& \leq & 1 + \sqrt{x} \log_2 x \\
& \in & O \left( \sqrt{x} \log x \right),
\ee
and hence
\be
x \in O \left( \sqrt{e_S(x-1)} \log e_S(x-1) \right).
\ee
This implies
\be
\left( \frac{x}{\log e_S(x-1)} \right)^2 \in O(e_S(x-1)).
\ee
and hence $x \log e_S(x) \log e_S(x-1) \in o(e_S(x-1))$. Therefore, the naive transcendence of the image of $(s_p)_{p \in \bP} \in \N^{\bP}$ in $\cA$follows from Corollary \ref{floor} (ii) applied to $S$.
\end{exm}

As an extension of the ceiling in Corollary \ref{floor} (ii), we consider a higher ceiling.

\begin{crl}
\label{k-th least}
Let $k \in \N_{> 0}$. We denote by $f \in \N^{\bP}$ the map assigning to each $p \in \bP$ the $k$-th least element of $S_{> p}$. If $x \log e_S(x) \log e_S(x-k) \in o(e_S(x-k))$, then the image of $f$ in $\cA$ is naively transcendental.
\end{crl}

\begin{proof}
By definition, $f$ is an $(S,k)$-rounding map. Therefore, The assertions immediately follows from Proposition \ref{sparse}.
\end{proof}

For a map $T \colon \N \to \N$, we denote by $T^k \colon \N \to \N$ the $k$-th iteration of $T$ for $k \in \N$. We also give a criterion of naive transcendence given by iteration of shifting. We note that $f$ in Corollary \ref{k-th least} coincides with $T^k |_{\bP}$ for the map
\be
T \colon \N & \to & \N \\
i & \mapsto & \min S_{> i}
\ee
whose restriction to $\bP$ is the map $(\min S_{> p})_{p \in \bP}$ in Corollary \ref{floor} (ii). In this sense, Corollary \ref{k-th least} gives a criterion of naive transcendence for iteration of a map.

\begin{exm}
Let $e \in \N_{> 1}$. We define a map $T \colon \N \to \N$ by
\be
i \mapsto \lfloor \sqrt[e]{i} \rfloor^e + 1.
\ee
For any $k \in \N_{> 0}$, the image of $T^k |_{\bP} \in \N^{\bP}$ in $\cA$ is naively transcendental. Indeed, $T^k - 1$ coincides with $f$ in Corollary \ref{k-th least} applied to the case where $S$ is the set of perfect $e$-th powers.
\end{exm}

\section{Montgomery Reduction}
\label{Montgomery Reduction}

Let $(N,R) \in \N_{> 0}^2$ with $N < R$ and $\gcd(N,R) = 1$. We denote by $\REDC_{N,R}$ the Montgomery reduction for $(N,R)$, i.e.\ the map
\be
\Z/N \Z & \to & \Z/N \Z \\
i + N \Z & \mapsto & iR' + N \Z,
\ee
where $R' \in \N_{< N}$ denotes the modular inverse of $R$ modulo $N$. In particular, $R' + N \Z = \REDC_{N,R}(1 + N \Z)$ represents $\REDC_{N,E}$ regarded as a $\Z/N \Z$-linear homomorphism.

\vs
We note that the assumption $N < R$ is used for the original implementation of $\REDC_{N,R}$, and we do not need it when we only observe its mathematical nature. Practically speaking, $N$ is assumed to be an odd number and $R$ is chosen to be a power of $2$ greater than $N$ but is bounded by a machine word size. Typically, $N$ is an odd number expressed by $32$ bits, and $R$ is $2^{32}$.

\vs
In this section, we consider Montgomery reduction for infinitely large primes. Let $(k,B,e) \in \N_{> 0} \times \N_{> 1} \times \N_{> 0}$ with $B^{e(p)-k} \leq p < B^{e(p)}$ for any $p \in \bP_{> B}$. We define a map $R \colon \bP \to \N$ by 
\be
R(p) \coloneqq B^{e(p)}.
\ee
For any $p \in \bP_{> B}$, $\REDC_{p,R(p)}$ makes sense by $p < R(p)$ and $\gcd(p,R(p)) = 1$.

\begin{thm}
\label{Montgomery}
For any $f \in \N^{\bP}$, if $f(p) + p \Z$ represents $\REDC_{p,R(p)}$ regarded as an $\Fp$-linear homomorphism for all but finitely many $p \in \bP_{> B}$, then the image of $f$ in $\cA$ is naively transcendental.
\end{thm}

In order to prove Theorem \ref{Montgomery}, we prepare a lemma.

\begin{lmm}
\label{Montgomery expression}
The image of $R$ in $\cA$ is naively transcendental.
\end{lmm}

\begin{proof}
The assertion immediately follows from Proposition \ref{sparse} applied to the case where $S$ is the set of powers of $B$.
\end{proof}

\begin{proof}[Proof of Theorem \ref{Montgomery}]
The image of $fR$ in $\cA$ is $1$. Therefore, the image of $f$ in $\cA$ is an inverse of that of $R$, which is naively transcendental by Lemma \ref{Montgomery expression}, and hence is naively transcendental.
\end{proof}

For $n \in \N_{> 0}$, we define {\it the digit flip of $n$ in base $B$} as
\be
\sum_{d=0}^{D} (B - 1 - a(d)) B^d
\ee
when $n$ is expressed as $\sum_{d=0}^{D} a(d) B^d$ with $D \in \N$ and $a \in \N^{D+1}$ satisfying $a(D) \neq 0$, and denote it by $\DF_B(n)$. We obtain the naive transcendence of digit flips of prime numbers.

\begin{crl}
\label{digit flip}
The image of $(\DF_B(p))_{p \in \bP}$ in $\cA$ is naively transcendental.
\end{crl}

\begin{proof}
For any $p \in \bP$, we have
\be
DF_B(P) = B^{D+1} - p - 1,
\ee
where $D$ denotes the greatest $d \in \N$ with $B^d \leq p$. This implies that $(\log_B(DF_B(p) + p + 1))_{p \in \bP}$ satisfies the condition of $e$ in Theorem \ref{Montgomery} with $k = 1$, and hence the image of $(DF_B(p) + p + 1)_{p \in \bP}$ in $\cA$ is naively transcendental by Lemma \ref{Montgomery expression}. This implies that the image of $(\DF_B(p))_{p \in \bP}$ in $\cA$ is naively transcendental.
\end{proof}

In particular, Corollary \ref{digit flip} applied to the case $B = 2$ implies the naive transcendence of bit flips of prime numbers (where we do not consider leading zeros of binary expressions).

\vspace{0.3in}
\addcontentsline{toc}{section}{Acknowledgements}
\noindent {\Large \bf Acknowledgements}
\vspace{0.2in}

\noindent
I thank all people who helped me to learn mathematics and programming. I also thank my family.

\end{document}